\documentclass[12pt,a4paper,reqno]{amsart}
\usepackage{amssymb}
\usepackage{latexsym}
\usepackage{exscale}
\usepackage{mathrsfs}

\usepackage{color}

\numberwithin{equation}{section}
\newtheorem{theorem}{Theorem}[section]
\newtheorem{lemma}[theorem]{Lemma}

\newtheorem{Remark}[theorem]{Remark}

\newtheorem{example}[theorem]{Example}

\newcommand{\Ba}[1]{\begin{array}{#1}}
\newcommand{\Ea}{\end{array}}
\newcommand{\Be}{\begin{equation}}
\newcommand{\Ee}{\end{equation}}
\newcommand{\Bea}{\begin{eqnarray}}
\newcommand{\Eea}{\end{eqnarray}}
\newcommand{\Beas}{\begin{eqnarray*}}
\newcommand{\Eeas}{\end{eqnarray*}}
\newcommand{\Benu}{\begin{enumerate}}
\newcommand{\Eenu}{\end{enumerate}}

\newcommand{\BR}{\begin{Remark} \em}
\newcommand{\ER}{\end{Remark}}
\newcommand{\BE}{\begin{example} \em}
\newcommand{\EE}{\end{example}}

\newcounter{reg}
\begin{document}
\allowdisplaybreaks

\title[Lebesgue-type inequalities for
quasi-greedy bases]{Lebesgue-type inequalities for quasi-greedy
bases}

\author{Eugenio Hern\'andez}

\address{Eugenio, Hern\'andez
\\
Departamento de Matem\'aticas
\\
Universidad Aut\'onoma de Madrid
\\
28049, Madrid, Spain}

\email{eugenio.hernandez@uam.es}

\begin{abstract}
We show that for quasi-greedy bases in real Banach spaces the error
of the thresholding greedy algorithm of order $N$ is bounded by the
best $N$-term error of approximation times a constant which depends
on the democracy functions and the quasi-greedy constant of the
basis.
\end{abstract}

\thanks{Research supported by Grant MTM2010-16518 (Spain).}

\date{\today}


\maketitle

\section{Introduction}\label{secIntroduc}

Let $(\mathbb{B},\|.\|_{\mathbb{B}})$ be a real Banach space with a
countable seminormalized basis  $\mathcal{B} = \{e_{k}: k \in
\mathbb{N}\}.$ Let $\Sigma_N\,, N=1,2,3,\dots$ be the set of all
$y\in \mathbb B$ with at most $N$ non-null coefficients in the
unique basis representation. For $x\in \mathbb B$, the {\bf $N$-term
error of approximation} with respect to $\mathcal B$ is
$$
  \sigma_N(x) = \sigma_N(x;\mathcal B, \mathbb B) :=
  \inf_{y\in\Sigma_N} \|x-y\|_{\mathbb B}\,, \quad N=1,2,3, \dots
$$
Given $x=\sum_{k\in \mathbb N} a_k(x)\,e_k \in \mathbb B\,,$ let
$\pi$ denote any bijection of $\mathbb N$ such that
\begin{eqnarray} \label{greedy}
  |a_{\pi(k)}(x)| \geq |a_{\pi(k+1)}(x)|\, \quad \mbox{for all} \
  k\in \mathbb N\,.
\end{eqnarray}
The {\bf thresholding greedy algorithm of order $N$} (TGA) is
defined by
$$
  G_N(x) = G_N^\pi (x;\mathcal B, \mathbb B) := \sum_{k=1}^N a_{\pi(k)}(x)
  e_{\pi(k)}\,.
$$
It is not always true that $G_N(x) \to x$ (in $\mathbb B$) as $N \to
\infty\,.$ A basis $\mathcal B$ is called {\bf quasi-greedy} if
$G_N(x) \to x$ (in $\mathbb B$) as $N \to \infty$ for all $x\in
\mathbb B\,.$ It turns out that this is equivalent (see Theorem 1 in
\cite {Woj2000}) to the existence of some constant $C$ such that
\begin{eqnarray}\label{1.4}
  \sup_N \|G_N(x)\|_{\mathbb B} \leq C \| x \|_{\mathbb B} \quad
  \mbox{for all} \ x\in \mathbb B\,.
\end{eqnarray}
It is convenient to define the {\bf quasi-greedy constant} $K$ to be
the least constant such that
$$
  \|G_N(x)\|_{\mathbb B} \leq K \|x\|_{\mathbb B} \quad \mbox{and}
  \quad \|x- G_N(x)\|_{\mathbb B} \leq K \|x\|_{\mathbb B}\,, \quad
  x\in \mathbb B\,.
$$

Given a basis $\mathcal B$ in a Banach space $\mathbb B$ a {\bf
Lebesgue-type} inequality is an inequality of the form
$$
  \|x- G_N(x)\|_{\mathbb B} \leq C\,
  v(N)\,\sigma_N(x)\,,\quad   x\in \mathbb B\,,
$$
where $v(N)$ is a nondecreasing function of $N$. For a survey of
Lebesgue-type inequalites see \cite{Tem2008} and the references
given there.

The purpose of this note is to find Lebesgue-type inequalities for
quasi-greedy basis in a Banach space.

For a seminormalized collection $\mathcal B = \{u_k\}_{k\in \mathbb
N}$ in a Banach space $\mathbb B$ the following quantities are
defined:
$$
  h_r(N) = \sup_{|\Gamma|=N} \Big\| \sum_{k\in \Gamma}
  u_k\Big\|\,, \quad h_l(N) = \inf_{|\Gamma|=N} \Big \|\sum_{k\in \Gamma}
  u_k\Big\|\,,
$$
and
\begin{eqnarray} \label{1.8}
  \mu(N) = \sup_{1\leq k \leq N} \frac{h_r(k)}{h_l(k)}\,, \quad N=1,2,3,
  \dots\,.
\end{eqnarray}
 These functions are implicit in earlier works on $N$-term
approximation and explicitly defined in \cite{KT2004}. The function
$\mu(N)$ is defined in \cite{Woj2000}. The functions $h_r$ and $h_l$
are called right and left democracy functions of $\mathcal B$ (see
\cite{GHM2008} and \cite{GHN2011}).

\vspace{.3cm}

Our main result is the following:

\begin{theorem}\label{Th1}
  Let $\mathcal{B} = \{e_k\}_{k=1}^\infty$ be a quasi-greedy basis
in a real  Banach space $\mathbb{B}$, and let $K$ be the
quasi-greedy constant of $\mathcal{B}$. Then for all $N=1,2,3,
\dots$ and all $x\in \mathbb B$,
$$
  \|x- G_N(x)\|_{\mathbb B} \lesssim 8K^5 \Big(\sum_{k=1}^N
  \mu(k)\frac{1}{k} \Big) \sigma_N(x)\,.
$$
\end{theorem}

Before proving Theorem \ref{Th1} we make some remarks about the
function
$$
  v(N) := \sum_{k=1}^N \mu(k)\frac{1}{k}\,.
$$
Obviously, $\mu$ is increasing so that $v(N) \lesssim \mu(N) \log
N\,.$ In some cases this inequality is an equivalence. For example
if $\mu(N) \approx C$ (that is $\mathcal B$ is democratic) then
$v(N) \approx \log N\,.$ In other cases the inequality can be
improved. It can be proved that if $\mathcal B$ is quasi-greedy,
$\mu$ is doubling, that is there exists a constant $D \geq 1$ such
that $\mu(2k) \leq D\mu(k)\,, k \in \mathbb N$ (see the Appendix).
Under this condition it is not difficult to prove that
$$
  v(N) \approx \sum_{k=1}^{\log_2 N} \mu(2^k)\,.
$$
Moreover, if we assume that $\mu$ has a positive dilation index,
that is $\mu \in \mathbb W_{+}$ in the terminology of
\cite{GHN2011}, by Lemma 2.1 in \cite{GHN2011} we have
$$
  v(N) \lesssim \mu(2^{\log_2 N}) = \mu (N)\,,
$$
so that in this situation we do not need the $\log N$ factor.

We prove Theorem \ref{Th1} in Section 2. Section 3 contains some
comments and open questions.

\section{Proof of Theorem \ref{Th1}}\label{Section2}

 We need the following result from \cite{DKKT2003}: let $\mathcal B
= \{e_k\}_{k=1}^\infty$ be a quasi-greedy bases with constant $K$ in
a real Banach space. For any finite set $\Gamma \subset \mathbb N$
and any real numbers $\{a_k\}_{k\in \Gamma}$ we have

\begin{eqnarray}\label{2.1}
   \frac{1}{4K^2}(\min_\Gamma |a_k|)\Big\|\sum_{k\in \Gamma}
   e_k\Big\|_\mathbb B \leq \Big\|
   \sum_{k\in \Gamma} a_k\,e_k\Big\|_\mathbb B \leq (2K)(\max_\Gamma
   |a_k|) \Big\|\sum_{k\in \Gamma} e_k\Big\|_\mathbb B
\end{eqnarray}
(see Lemma 2.1 and Lemma 2.2 in \cite{DKKT2003}).

\begin{lemma}\label{lema2.1}
Let $\mathcal B$ and $\mathbb B$ as in Theorem \ref{Th1}. Suppose
that there exists $C_1 > 0$ such that for all $\Gamma \subset
\mathbb N$ finite
\begin{eqnarray}\label{2.2}
   \Big\| \sum_{k\in \Gamma} e_k\Big\|_\mathbb B \leq C_1 \eta
   (|\Gamma|)
\end{eqnarray}
for some $\eta$ increasing and doubling (for example, $\eta(N) =
h_r(N)$ and $C_1 = 1$). Then there exists $C = C_\eta$ such that for
any $x = \sum_{k\in \mathbb N} a_k(x) e_k \in \mathbb B$
\begin{eqnarray}\label{2.3}
   \|x\|_{\mathbb B} \leq 2K\,C_\eta  \sum_{k=1}^\infty a_k^*(x)
   \eta(k) \frac{1}{k}\,,
\end{eqnarray}
where $\{a_k^*(x)\}$ is a decreasing rearrangement of $\{|a_k(x)|\}$
as in (\ref{greedy}).
\end{lemma}

\begin{proof}
Let $\pi$ be a bijection of $\mathbb N$ that gives $\{a_k^*(x)\}$,
that is $\{a_k^*(x)\}= |a_{\pi_{k}}(x)|\,.$ Since $\mathcal B$ is
quasi-greedy
$$
  \lim_{N\to \infty} \sum_{k=1}^N a_{\pi_{k}}(x) e_{\pi_{k}} \to
  x \,(\, \mbox{convergence in}\  \mathbb B)\,.
$$
Thus
\begin{eqnarray*}
    \|x\|_{\mathbb B} &=&
    \Big\|\sum_{k=1}^{\infty}a_{\pi(k)}(x)\,e_{\pi(k)}\Big\|_{\mathbb{B}}
    =\Big\|\sum_{j=0}^{\infty}\sum_{2^{j}\leq k <2^{j+1}}
    a_{\pi(k)}(x)\,e_{\pi(k)}\Big\|_{\mathbb{B}}\,\\
    &\leq &\sum_{j=0}^{\infty} \Big\| \sum_{2^{j}\leq k <2^{j+1}}
    a_{\pi(k)}(x)\,e_{\pi(k)}\Big\|_{\mathbb{B}}\,.
\end{eqnarray*}
We now use first the right hand side inequality of (\ref{2.1}) and
then condition (\ref{2.2}) to deduce
\begin{eqnarray*}
    \|x\|_{\mathbb B} &=&
    \sum_{j=0}^{\infty} (2K) |a_{\pi(2^j)}(x)|\Big\| \sum_{2^{j}\leq k <2^{j+1}}
    e_{\pi(k)}\Big\|_{\mathbb{B}}
    \leq (2K) C_1\sum_{j=0}^{\infty} |a_{\pi(2^j)}(x)| \eta(2^j)\\
    &= & (2K) C_1\sum_{j=0}^{\infty} a_{2^j}^* \eta(2^j) \,.
\end{eqnarray*}
Inequality (\ref{2.3}) follows since $\eta$ is doubling and
increasing.
\end{proof}

\begin{lemma}\label{lema2.2}
Let $\mathcal B$ and $\mathbb B$ as in Theorem \ref{Th1}. Suppose
that there exists $C_2 > 0$ such that for all $\Gamma \subset
\mathbb N$ finite
\begin{eqnarray}\label{2.4}
   \frac{1}{C_2} \eta(|\Gamma|) \leq\Big\| \sum_{k\in \Gamma} e_k\Big\|_\mathbb B
\end{eqnarray}
for some function $\eta$ (for example, $\eta(N) = h_l(N)$ and $C_2 =
1$). Then for any $x = \sum_{k\in \mathbb N} a_k(x) e_k \in \mathbb
B$
\begin{eqnarray}\label{2.5}
   [\sup_\Gamma a_k^*(x) \eta(k)] \leq C_2 (4K^3) \|x\|_{\mathbb
   B}\,,
\end{eqnarray}
where $\{a_k^*(x)\}$ is a decreasing rearrangement of $\{|a_k(x)|\}$
as in (\ref{greedy}).
\end{lemma}

\begin{proof}
Let $\pi$ be as in the proof of Lemma \ref{lema2.1}. For any $k\in
\mathbb N$ we use condition (\ref{2.4}) and then the left hand side
inequality of (\ref{2.1}) to obtain
\begin{eqnarray*}
   |a_{\pi(k)}(x)|\eta(k) \leq C_2 |a_{\pi(k)}(x)|
   \Big\|\sum_{j=1}^k e_{\pi(j)}\Big\| \leq C_2 (4K^2)
   \Big\|\sum_{j=1}^k a_{\pi(j)} e_{\pi(j)}\Big\|\,.
\end{eqnarray*}
We use (\ref{1.4}) to deduce $|a_{\pi(k)}(x)| \eta(k) \leq C_2
(4K^3) \|x\|_{\mathbb B}\,.$ The result follows by taking the
supremum on $k\in \Gamma$.
\end{proof}

For $\Gamma \subset \mathbb N$ and $x=\sum_{k\in \mathbb N} a_k(x)
e_k \in \mathbb B$ define the {\bf projection operator over
$\Gamma$} as
$$
   S_\Gamma (x) := \sum_{k\in\Gamma} a_k(x) e_k \,.
$$

\begin{lemma}\label{lema2.3}
Let $\mathcal B$ and $\mathbb B$ as in Theorem \ref{Th1}. For
$\Gamma \subset \mathbb N$ finite
\begin{eqnarray*}
   \|S_\Gamma (x)\|_{\mathbb B} \lesssim (8K^4)\Big( \sum_{k=1}^{|\Gamma|} \mu(k)
   \frac{1}{k}\Big) \|x\|_\mathbb B\,.
\end{eqnarray*}
\end{lemma}

\begin{proof}
Apply Lemma \ref{lema2.1} with $\eta(N) = h_r(N)$ to obtain
\begin{eqnarray*}
    \|S_\Gamma(x)\|_{\mathbb B} &=&
    \Big\|\sum_{k\in \Gamma} a_{k}(x)\,e_{k}\Big\|_{\mathbb{B}}
    \lesssim (2K) \sum_{k=1}^{|\Gamma|} a_{k}^*(x)\,h_r(k)\frac{1}{k}\,\\
    &\leq &  (2K) \sum_{k=1}^{|\Gamma|} a_{k}^*(x)\frac{h_r(k)}{h_l(k)}
    \,h_l(k)\frac{1}{k}\,\\ &\leq & (2K) [\sup_k a_k^*(x) h_l(k)]
    \sum_{k=1}^{|\Gamma|}\mu(k)\frac{1}{k}\,.
\end{eqnarray*}
Use now Lemma (\ref{2.2}) with $\eta(k) = h_l(k)$ to deduce the
result.
\end{proof}

\vspace{.3cm}

We now prove Theorem \ref{Th1}. The proof follows arguments used in
\cite{KT2004}, \cite{TYY2011a} and \cite{TYY2011b} that were
presented by V. N. Temlyakov at the \textit{Concentration week on
greedy algorithms in Banach spaces and compressed sensing} held on
July 18-22 at Texas A\&M University.

Take $\epsilon >0$ and $N=1,2,3, \dots.$ Choose $p_N(x) = \sum_{k\in
P} b_k\, e_k$ with $|P|=N$ such that
\begin{eqnarray}\label{2.8}
    \|x - p_N(x)\|_{\mathbb B} \leq \sigma_N(x) + \epsilon\,.
\end{eqnarray}
Let $\Gamma$ be the set of indices picked by the thresholding greedy
algorithm after $N$ iterations, that is
\begin{eqnarray*}
    G_N(x) = \sum_{k\in \Gamma} a_k(x)\, e_k \,, \quad |\Gamma|=N\,.
\end{eqnarray*}
We have, from Lemma \ref{lema2.3} and (\ref{2.8})
\begin{eqnarray}
    & &\|x- G_N(x)\|_{\mathbb B} \leq  \|x - p_N(x)\|_{\mathbb
    B} + \|p_N(x) - S_P(x)\|_{\mathbb B} + \|S_P(x) - S_\Gamma (x)\|_{\mathbb B}
    \,\nonumber \\
    &= & \|x - p_N(x)\|_{\mathbb B} + \|S_P(x-p_N(x))\|_{\mathbb B} +
    \|S_P(x) - S_\Gamma (x)\|_{\mathbb B}
    \,\nonumber \\
    &\lesssim & \Big[1+8K^4 \Big(\sum_{k=1}^{|\Gamma|}
    \mu(k)\frac{1}{k}\Big)\Big] \|x - p_N(x)\|_{\mathbb B} +
    \|S_{P\setminus \Gamma}(x) - S_{\Gamma \setminus P} (x)\|_{\mathbb B}
    \,\nonumber \\
    &\leq & \Big[1+8K^4 \Big(\sum_{k=1}^{|\Gamma|}
    \mu(k)\frac{1}{k}\Big)\Big] (\sigma_N(x) + \epsilon) +
    \|S_{P\setminus \Gamma}(x)\|_{\mathbb B} + \|S_{\Gamma \setminus P} (x)\|_{\mathbb
    B}\,. \label{2.10}
\end{eqnarray}
It is not difficult to find an upper bound for $\|S_{\Gamma
\setminus P}(x)\|_{\mathbb B}\,.$ Since $p_N(x)$ is supported in $P$
we have $S_{\Gamma \setminus P}(p_N(x)) = 0$. By Lemma \ref{lema2.3}
\begin{eqnarray}
    \|S_{\Gamma \setminus P}(x)\|_{\mathbb B} &=&
    \|S_{\Gamma \setminus P}(x-p_N(x))\|_{\mathbb B} \lesssim
    (8K^4) \Big(\sum_{k=1}^{|\Gamma\setminus P|} \mu(k)\frac{1}{k}\Big)
    \|x-p_N(x)\|_{\mathbb B}   \,\nonumber \\
    &\leq & (8K^4) \Big(\sum_{k=1}^{N} \mu(k)\frac{1}{k}\Big)
    [\sigma_N(x) + \epsilon]\,. \label{2.11}
\end{eqnarray}
The bound for $\|S_{P \setminus \Gamma}(x)\|_{\mathbb B}$ is more
delicate. Use Lemma \ref{lema2.1} with $\eta(N) = h_r(N)$ to write
$$
   \|S_{P \setminus \Gamma}(x)\|_{\mathbb B} =
   \Big\|\sum_{k\in P \setminus \Gamma} a_k(x) \, e_k \Big\|_{\mathbb B} \leq
   (2K) \sum_{k=1}^{|P \setminus \Gamma|} a_k^*(S_{P \setminus
   \Gamma}(x)) h_r(k)\frac{1}{k}\,.
$$
If $k\in \Gamma\setminus P$ and $s\in P \setminus \Gamma$ we have
$a_s^*(S_{P \setminus \Gamma}(x))\leq a_k^*(S_{\Gamma \setminus
P}(x))$ by construction of the thresholding greedy algorithm since
$\min_\Gamma |a_k(x)| \geq \max_{\mathbb N \setminus \Gamma}
|a_k(x)|\,.$ Also, since $|P|=N=|\Gamma|$ we have $|P\setminus
\Gamma| = |\Gamma \setminus P|\,.$ Thus
$$
   \|S_{P \setminus \Gamma}(x)\|_{\mathbb B} \leq
   (2K) \sum_{k=1}^{|\Gamma \setminus P|} a_k^*(S_{\Gamma \setminus
   P}(x)) h_r(k)\frac{1}{k}\,.
$$
We use that $S_{\Gamma \setminus P}(p_N(x)) = 0$ and Lemma
\ref{lema2.2} with $\eta = h_l$ to obtain
\begin{eqnarray*}
    \|S_{P \setminus \Gamma}(x)\|_{\mathbb B} &\lesssim &
    (2K) \sum_{k=1}^{|\Gamma \setminus P|}
    a_k^*(S_{\Gamma \setminus P}(x-p_N(x))) h_r(k)\frac{1}{k} \\
    &= & (2K) \sum_{k=1}^{|\Gamma \setminus P|}
    a_k^*(G_{|\Gamma \setminus P|}(x-p_N(x)))\frac{h_r(k)}{
    h_l(k)} h_l(k)\frac{1}{k}  \\
    &\leq & (2K) [\sup_k a_k^*(G_{|\Gamma \setminus P|}(x-p_N(x)))h_l(k)]
    \Big(\sum_{k=1}^N \mu(k)\frac{1}{k} \Big)\\
    &\leq & (2K)(4K^3) \| G_{|\Gamma \setminus P|}(x-p_N(x))\|_{\mathbb B}
    \Big(\sum_{k=1}^N \mu(k)\frac{1}{k} \Big) \,.
\end{eqnarray*}
We use now that $\mathcal B$ is a quasi-greedy basis to write
\begin{eqnarray}
    \|S_{P \setminus \Gamma}(x)\|_{\mathbb B} &\lesssim &
    (8K^5) \| x-p_N(x)\|_{\mathbb B}
    \Big(\sum_{k=1}^N \mu(k)\frac{1}{k} \Big) \nonumber \\
    &\leq & (8K^5)  (\sigma_N(x)+\epsilon)
    \Big(\sum_{k=1}^N \mu(k)\frac{1}{k} \Big)\ \,. \label{2.12}
\end{eqnarray}
Replacing (\ref{2.11}) and (\ref{2.12}) in (\ref{2.10}), and letting
$\epsilon \to 0$ we obtain the result stated in Theorem \ref{Th1}.

\section{Comments and questions}\label{Section3}

{\bf 3.1}. Let $\mathcal B$ be a seminormalize quasi-greedy basis in
a real Hilbert space $\mathbb H.$ Since $\mathcal B =
\{e_k\}_{k=1}^\infty$ is unconditional for constant coefficients
(see Proposition 2 in \cite{Woj2000}) it follows from Kintchine's
inequality that
$$
    \Big\| \sum_{k\in\Gamma} e_k\Big\|_{\mathbb H} \approx
    \sqrt{|\Gamma|}\,.
$$
Thus, in this case we can take $\eta(N)=N^{1/2}$ in Lemma
\ref{lema2.1} and Lemma \ref{lema2.2}, giving us Theorem 3 from
\cite{Woj2000}.

\vspace{.3cm}

{\bf 3.2}. Let $\mathcal B = \{e_k\}_{k=1}^\infty$ be a quasi-greedy
basis in $L^p(\mathbb T^d)\,.$ If $2 \leq p < \infty$ the space
$L^p(\mathbb T^d)$ has type 2 and cotype $p$. Thus
\begin{eqnarray} \label{3.1}
   C_p' |\Gamma|^{1/p} \leq \Big\| \sum_{\Gamma} e_k \Big\|
   \leq C_p |\Gamma|^{1/2}\,, \quad \Gamma \subset \mathbb N\,.
\end{eqnarray}
Taking $\eta(N) \thickapprox N^{1/2}$ in Lemma \ref{lema2.1} and
$\eta(N) \thickapprox N^{1/p}$ in Lemma \ref{lema2.2} we obtain
Theorem 11 form \cite{TYY2011a} for the case $2\leq p < \infty\,.$

The case $1<p \leq 2$ of Theorem 11 from \cite{TYY2011a} is obtained
by observing that for this range of $p$'s the space  $L^p(\mathbb
T^d)$ has type $p$ and cotype 2, so that
\begin{eqnarray} \label{3.2}
   C_p' |\Gamma|^{1/2} \leq \Big\| \sum_{\Gamma} e_k \Big\|
   \leq C_p |\Gamma|^{1/p}\,, \quad \Gamma \subset \mathbb N\,.
\end{eqnarray}

\vspace{.3cm}

{\bf 3.3}. The proofs of Lemmata \ref{lema2.1} and \ref{lema2.2}
follow the pattern of the proofs of 1. $\Rightarrow$ 2. in Theorem
3.1 and Theorem 4.2 from \cite{GHN2011} for the limiting case
$"\alpha = 0"\,.$

\vspace{.3cm}

{\bf 3.4}. As in \cite{Woj2000} write, for $N=1,2,3, \dots $
$$
   e_N (\mathbb B) = e_N := \sup_{x\in \mathbb B}
   \frac{\|x-G_n(x)\|_{\mathbb B}}{\sigma_N(x)}\,, \quad (\frac{0}{0}
   = 1)\,.
$$
Theorem \ref{Th1} shows that for a quasi-greedy basis in a real
Banach space
$$
   e_N \leq C \Big(\sum_{k=1}^N \mu(k)\frac{1}{k} \Big) \lesssim
   \mu(N) \log N\,, \quad N\in \mathbb N\,.
$$
For unconditional bases, Theorem 4 from \cite{Woj2000} shows that
\begin{eqnarray} \label{3.5}
  e_N \approx \mu(N)\,, \quad N\in \mathbb N\,.
\end{eqnarray}
The same argument that proves (\ref{3.5}) can be used to prove the
following result: for a quasi-greedy basis $\mathcal B$ in a real
Banach space $\mathbb B$
\begin{eqnarray} \label{3.6}
  \tilde {e}_N \approx \mu(N)\,, \quad N\in \mathbb N\,.
\end{eqnarray}
were
$$
   \tilde{e}_N (\mathbb B) = \tilde{e}_N := \sup_{x\in \mathbb B}
   \frac{\|x-G_n(x)\|_{\mathbb B}}{\tilde \sigma_N(x)}\,, \quad (\frac{0}{0}
   = 1)\,.
$$
and
$$
  \tilde \sigma_N(x) = \tilde \sigma_N(x;\mathcal B, \mathbb B) :=
  \inf \{\|x- \sum_{k\in \Gamma} a_k(x)\,e_k\|_{\mathbb B}\, :
  |\Gamma|\leq N\}
$$
is the  \textit{expansional} best approximation to $x = \sum_{k\in
\mathbb N} a_k(x)\,e_k \in \mathbb B\,.$

Since $\sigma_N(x) \leq \tilde \sigma_N (x)$, for a quasi-greedy
basis we have by (\ref{3.6}) and Theorem \ref{Th1}
\begin{equation} \label{3.7}
   \mu(N) \lesssim \tilde e_N (\mathcal B) \leq e_N (\mathcal B) \lesssim
   \mu(N) \log N\,.
\end{equation}
By the comments that follow the statement of Theorem \ref{Th1} if
$\mu$ has positive dilation index, $\mu(N) \lesssim \tilde e_N
(\mathcal B) \leq e_N (\mathcal B) \lesssim    \mu(N) \,.$ The last
inequality in (\ref{3.7}) was proved in \cite{Woj2000} for the
Hilbert space case (see the Remark that follows the proof of Theorem
5 in \cite{Woj2000}).

\vspace{.3cm}

{\sc Question 1.} Is the inequality on the right hand side of
(\ref{3.7}) sharp? That is, is it possible to find a quasi-greedy
basis $\mathcal B$ such that $e_N (\mathcal B) \approx
   \mu(N) \log N\,? $ This question appears in \cite{TYY2011a} for
the Hilbert space case (see paragraph that follows Theorem 10 in
\cite{TYY2011a}).

\vspace{.3cm}

{\sc Question 2.} Is it true that for a quasi-greedy basis $\tilde
\sigma_N (x) \lesssim \sigma_N(x) \log N\,?$ If the answer is "yes"
then by (\ref{3.6}) we will have $ e_N (\mathcal B)  \lesssim \tilde
e_N (\mathcal B) \log N \lesssim \mu(N) \log N\,,$ given another
proof of the right hand side of (\ref{3.7}).

\vspace{.3cm}

{\bf 3.5}. For a quasi-greedy basis $\mathcal B =
\{e_k\}_{k=1}^\infty$  in $L^p(\mathbb T^d)\,,$ inequalities
\ref{3.1} and \ref{3.2} (or type and cotype properties of
$L^p(\mathbb T^d)$) show that $\mu(N) \lesssim N^{|\frac{1}{p} -
\frac{1}{2}|}\,.$ By the comments that follow the statement of
Theorem \ref{Th1}, if $p\neq 2$ and $1 < p <\infty$, $e_N(\mathbb
B)\lesssim N^{|\frac{1}{p} - \frac{1}{2}|}\,,$ proving Theorem 1.1
from \cite{TYY2011b}. (Notice that $w(N):= N^{|\frac{1}{p} -
\frac{1}{2}|}\,$ has positive dilation index if $p\neq 2$.) For
$p=2$ we have $e_N(\mathcal B) \lesssim \log N$ by Theorem
\ref{Th1}.

Consider now the trigonometric system $\mathcal T^d = \{e^{ikx} :
k\in \mathbb Z^d\}$ in $L^p(\mathbb T^d)\,, 1 \leq p \leq \infty$
(here $L^\infty(\mathbb T^d)$ is $C(\mathbb T^d)$, the set of
continuous functions in $\mathbb T^d$). It is proved in
\cite{Tem1998} (Theorem 2.1) that
$$
   e_N(\mathcal T^d,L^p(\mathbb T^d)) \lesssim N^{|\frac{1}{p} -
   \frac{1}{2}|}\,, \quad 1\leq p \leq \infty\,.
$$

\vspace{.3cm}

{\sc Question 3.} (Asked by V. N. Temlyakov at the
\textit{Concentration week on greedy algorithms in Banach spaces and
compressed sensing} held on July 18-22 at Texas A\&M University.)

a) Characterize those systems $\mathcal B$ in $L^p(\mathbb T^d)$,
$1\leq p \leq \infty\,,$ such that $e_N(\mathcal T^d,L^p(\mathbb
T^d)) \lesssim N^{|\frac{1}{p} - \frac{1}{2}|}\,,$ $N\in \mathbb
N\,.$ Notice that if $1 < p \neq 2  < \infty\,,$ the
characterization must be satisfy by $\mathcal T^d$ as well as any
quasi-greedy basis.

More generally,

b) Let $v(N)$ be an increasing function of $N$. Characterize those
systems $\mathcal B$ in a Banach space $\mathbb B$ for which $
e_N(\mathcal B, \mathbb B) \lesssim v(N)\,.$

\section{Appendix}\label{Section4}

\begin{lemma}\label{lema4.1}
If $\mathcal B$ is a quasi-greedy basis in a Banach space $\mathbb
B$, the function $\mu$ defined in (\ref{1.8}) is doubling.
\end{lemma}

\begin{proof}
It is proved in \cite{Woj2000} and \cite{DKKT2003} that for a
quasi-greedy basis $\mathbb B = \{e_k\}_{k=1}^\infty$ with
quasi-gredy constant $K$, if $B\subset A \subset \mathbb N$ (finite
sets) then
\begin{eqnarray} \label{4.1}
\Big\| \sum_{k\in B} e_k \Big\|_{\mathbb B} \leq K \Big\| \sum_{k\in
A} e_k \Big\|_{\mathbb B}\,.
\end{eqnarray}
We have to prove that $\mu(2N) \leq D \mu(N)$ for some $D$
independent of $N$. Since $\mu(2N)$ is defined as a supremum over
the finite set $1 \leq k \leq 2N\,,$ there exists $k_0 \leq 2N$ such
that $\mu(2N) = h_r(k_0)/h_l(k_0)\,.$ Notice that $h_r$ is doubling
with doubling constant 2 by the triangle inequality.

Suppose first that $k_0=2s \leq 2N$ is even. From (\ref{4.1}) we
deduce $h_l(s) \leq K h_l(2s)\,.$ Hence
$$
   \mu(2N) = \frac{h_r(2s)}{h_l (2s)} \leq (2K) \frac{h_r(s)}{h_l(s)}
   \leq (2K) \mu(N)
$$
since $s\leq N\,.$

Assume now that $k_0 = 2s+1$ is odd. Since $2s+1=k_0 \leq 2N$ we
deduce $s\leq N-\frac{1}{2}\,,$ and since $s$ is an integer $s\leq
N-1\,.$ From (\ref{4.1}) we deduce $h_r(2s+1) \leq K\,h_r(2s+2)$ and
$h_l(s+1) \leq K h_l(2s+1)\,.$ Hence
$$
   \mu(2N)= \frac{h_r(2s+1)}{h_l(2s+1)} \leq K^2\, \frac{h_r(2s+2)}{h_l(s+1)}
    \leq 2 K^2\, \frac{h_r(s+1)}{h_l(s+1)} \leq (2K^2)\mu(N)
$$
since $s+1 \leq N\,.$
\end{proof}

\vspace{.3cm}

{\bf Acknowledgements}. This work started when the author
participated in the \textit{Concentration week on greedy algorithms
in Banach spaces and compressed sensing} held on July 18-22 at Texas
A\&M University. I would like to express my gratitude to the
Organizing Committee for the invitation to participate in this
meeting.

\end{document}